\documentclass[11pt,a4paper]{amsart}

\usepackage{graphicx}
\usepackage{amssymb,amsmath,amsthm}
\usepackage{pdfsync}
\usepackage[latin1]{inputenc}
\usepackage{enumerate}
\usepackage{hyperref}

\hypersetup{
    colorlinks=true, 
    linktoc=all,     
    linkcolor=blue,  
}

 \newtheorem{thm}{Theorem}[section]
 \newtheorem{cor}[thm]{Corollary}
 
 \newtheorem{prop}[thm]{Proposition}
 \theoremstyle{definition}
 \newtheorem{defn}[thm]{Definition}
 \theoremstyle{remark}
 \newtheorem{rem}[thm]{Remark}
 \newtheorem{ex}{Example}
 \numberwithin{equation}{section}

\def\R{\mathbb R}

\newcommand{\plans}{\mathcal{A}_{3,2} }
\newcommand{\rectes}{\mathcal{A}_{3,1} }
\newcommand{\grass}{\mathrm{Gr}(3,2)}
\newcommand{\euc}{\mathbb{E}^{3}}

\begin{document}

\title[Integral Geometry of pairs of planes] {Integral Geometry of pairs of planes}

\author[J.\ Cuf\'{\i}, E.\ Gallego and A.\ Revent\'os]{Juli\`a Cuf\'{\i}, Eduardo Gallego and Agust\'{\i} Revent\'os}

\address{%
Departament de Matem\`atiques \\
Universitat Aut\`{o}noma de Barcelona\\
08193 Bellaterra, Barcelona\\
Catalonia}
\email{jcufi@mat.uab.cat, egallego@mat.uab.cat, agusti@mat.uab.cat}

\thanks{The authors were partially supported by grants 2017SGR358, 2017SGR1725 (Generalitat de Catalunya) and PGC2018-095998-B-100 (FEDER/MICINN)}


\begin{abstract}
We deal with integrals of invariant measures of pairs of planes in euclidean space $\euc$ as considered by Hug and Schneider.~In this paper we   express some of these integrals in terms of functions of the visual angle of a convex set.~As a consequence of our results we evaluate  the deficit in a Crofton-type inequality due to Blashcke.
\end{abstract}

\maketitle

\section{Introducction}
%
%
%
%
%
The main goal of this paper  is to study integrals of invariant measures with respect to euclidean motions in the euclidean space $\euc$, extended to 
the set of pairs of planes meeting a compact convex set.~To carry out this objective we express these integrals in terms of functions of the dihedral visual angle  of the convex set from a line 
 and integrate  them with respect to an invariant   measure in the space of lines.

The first known formula involving the visual angle of a convex set in the euclidean plane $\mathbb{E}^{2}$
is Crofton's formula given in \cite{Crofton}.  Other results in this direction were  obtained by Hurwitz (\cite{Hurwitz1902}), Masotti    (\cite{masotti2}) and others, in which the use of  Fourier series is  the main tool. Recently the authors (\cite{Cufi2019}, \cite{mathematika}) have dealt with a general type of integral formulas from the point of view of Integral Geometry. 

When trying to generalize these results to higher dimensions the role played by Fourier series in the case of the plane has to be replaced by the use of spherical harmonics.  In this sense Theorem \ref{2909b} plays an important role.~After stating and proving this result in dimension $3$ we were aware of the paper by Hug and Schneider~\cite{hugsch} where a more general result  in any dimension is proved.  In fact the present paper can be considered in some sense as a complement to \cite{hugsch}, the novelty been the introduction of the dihedral visual angle. 

In Proposition \ref{mesuresinv} we give a characterization of invariant measures in the space 
of pairs of planes. These will be the kind of measures considered along the  paper.  

In section 4, using Hug-Schneider's Theorem \cite[p. 349]{hugsch}, we give an expression for the integral of the sinus of 
the dihedral visual angle of pairs of planes meeting a given compact convex set $K$ in terms of geometrical properties of $K$, see formula~\eqref{lambdessinus}; also  we characterize the compact convex sets of constant width in terms of invariant measures given by Legendre polynomials in Proposition \ref{Kampconst}.

In section 5  we assign to any invariant measure on the space of pairs of planes an appropriate function of the dihedral visual angle of a given convex set.~The integral of this function with respect to the measure on the space of lines 
gives the integral of the above  measure  extended to those planes meeting the convex set. This result is given in Theorem 
\ref{1310c}. Then we relate this result to Blaschke's work \cite{Blaschke2007}. If $K$ is a convex set of mean curvature  $M$ and area of its boundary $F$,  Blaschke proves the formula 
$$\int_{G\cap K=\emptyset}(\omega^{2}-\sin^{2}\omega)\,dG=2M^{2}-\frac{\pi^{3}F}{2}$$
  where $\omega=\omega(G)$ is the dihedral visual angle of $K$ from the line  $G$.
This equality reveals the significance of the function of the visual angle $\omega^{2}-\sin^{2}\omega$. One can ask what role does it play the function $\omega-\sin\omega$; this function, interpreting $\omega$ as the visual angle in the plane  is significant thanks to Crofton formula. In \cite{Blaschke2007} the inequality 
$$\int_{G\cap K=\emptyset}(\omega-\sin\omega)\,dG\geq \frac{\pi}{4}(M^{2}-2\pi F)$$
  is stablished. Here we provide a simple formulation of the deficit in this inequality by means of Theorem \ref{1903} where it is proved that
  $$\displaystyle \int_{G\cap K=\emptyset}(\omega-\sin\omega)dG={\pi\over 4}(M^{2}-2\pi F)+\pi\sum_{n=1}^{\infty}{\Gamma(n+1/2)^{2}\over \Gamma(n+1)^{2}}\lVert \pi_{2n}(p)\rVert^{2},$$
  whith $\pi_{2n}(p)$ the projection of the support function $p$ of $K$ on the vector space of spherical harmonics of degree $2n$.
  
In section 6 we give a formulation of Theorem \ref{1310c} in terms of Fourier series of the function of the visual angle assigned to an  invariant measure. As a consequence one obtains that the integral of any invariant measure in the space of pairs of planes extended to those meeting a compact convex set $K$ is an infinite linear combination of integrals of
 even powers of the sine of the visual angle of $K$. From  this we  exhibit in Proposition \ref{2403b} a simple family of polynomial functions that are in some sense a basis for the integrals considered in Theorem \ref{2909b}. In fact   every invariant integral can be written as an infinite linear combination of integrals with respect to the  invariant measures given by those polynomial functions. 

In section 7, motivated by the role played by the integrals of even powers of the sine of the visual angle of a compact set $K$ we compute these integrals in terms of the expansion in spherical harmonics  of the support function of $K$.
\section{Preliminaries}
\subsection*{Support function}

The  \emph{support function} of a compact convex set $K$  in the euclidean space $ \euc$ is defined as 
$$
p_{K}(u)=\mathrm{sup}\{\langle x,u\rangle : x\in K\}
$$
for $u$ belonging to the unit sphere $S^{2}$.
If the origin $O$ of $\euc$ is an interior point of $K$ then the number $p_{K}(u)$ is the distance from the origin to the support plane  of $K$ in the direction given by $u$.~The \emph{width} $w$ of $K$ in a direction $u\in S^{2}$ is $w(u)=p_{K}(u)+p_{K}(-u).$ 


From now on we will write $p(u)=p_{K}(u)$ and will assume that $p(u)$ is of class~$C^{2}$; in this case we shall say that the boundary of $K$, $\partial K$, is of class ${\mathcal C}^{2}$. 

\subsection*{Spherical harmonics}
Let us recall that a spherical harmonic of order $n$ on the unit sphere $S^{2}$ is the restriction to $S^{2}$ of an harmonic homogeneous polynomial of degree $n$. It is known that
every continuous function on $S^{2}$ can be uniformly approximated by finite sums of spherical harmonics (see for instance \cite{groemer}).

More precisely, the function  $p(u)$ can be written in terms of spherical harmonics as 
\begin{eqnarray}\label{1607}p(u)
=\sum_{n=0}^{\infty}\pi_{n}(p)(u),\end{eqnarray}
where $\pi_{n}(p)$  is the projection of the support function  $p$ on the vector space of spherical harmonics of degree $n$.~An orthogonal  basis of this space is given in terms of the longitude $\theta$ and the colatitude $\varphi$ in $S^{2}$ by 
\begin{equation*}
\{\cos(j\theta)(\sin\varphi)^{j}\,P_{n}^{(j)}(\cos\varphi), \quad
\sin(j\theta)(\sin\varphi)^{j}\,P_{n}^{(j)}(\cos\varphi):\quad 0\leq j\leq n\}
\end{equation*}
where   
$P_{n}^{(j)}$ denotes the $j$th derivative of the $n$th Legendre polynomial $P_{n}$, see \cite{groemer}.

It  can be seen  that $\pi_{0}(p)= \mathcal{W}/2=M/4\pi$
where $\mathcal{W}=1/4\pi\int_{S^{2}}w(u)	\, du$ is the {\em mean width} of $K$, and $M$ is the {\em mean curvature} of $K$ (cf. \cite{groemer}). 
It is clear that $\pi_{0}(p)$ is invariant under euclidean motions and that $\pi_{1}(p)$ is not.  It is known that $\pi_{n}(p)$ is also invariant for every $n\not =1$  (cf. \cite{Schneider2013}).

As $w(u)=p(u)+p(-u),$ one can easily check that  $K$ has constant width if and only if $\pi_{n}(p)=0$ for $n\neq 0$ even.

%
%
%
\medskip	

We recall now the Funk-Hecke theorem which gives the value of the integral over the sphere of a given function multiplied  by a spherical harmonic.~We restrict ourselves to the case of dimension $3$.
\begin{thm}[Funk-Hecke]\label{FH}
If $F:[-1,1]\longrightarrow \R$ is a bounded measurable function and $Y_{n}$ is a spherical harmonic of order $n$, then
$$\int_{S^{2}}F(\langle u,v\rangle)Y_{n}(v)d\sigma(v)=\lambda_{n}Y_{n}(u), \quad u\in S^{2}$$
with
$$\lambda_{n}=2\pi \int_{-1}^{1}F(t)P_{n}(t)\,dt,$$
where $P_{n}$ is the Legendre polynomial of degree $n$.
\end{thm}

\subsection*{Measures in the space of planes}
The space of planes $ \mathcal {A} _ {3,2} $ in $ \euc$ is a homogeneous space of the group of  isometries  of $\euc $. It can also be considered  as a line bundle 
$ \pi:\plans{\longrightarrow} \grass $ 
where $ \grass $ is the Grassmannian of planes through the origin in $\euc$ and $ \pi (E) $ is the plane parallel  to  $ E $  through the origin. The fiber on $E_{0}\in\grass$ is identified with $\langle E_{0}\rangle ^ {\perp}$.
Each plane $ E \in \plans $ is then uniquely determined by the pair $ (\pi (E), E \cap \langle \pi (E) \rangle ^ {\perp}). $ Every pair $ (E_ {0}, p) \in \grass \times \R ^ {3} $ determines an element $ E_ {0} + p \in \plans $.
\smallskip	

We shall also consider the space of affine lines $\rectes$ in $\euc$; it is a vector bundle $\pi:\rectes \longrightarrow \mathrm{Gr(3,1)}$ where $ \mathrm{Gr}(3,1) $ is the Grassmannian of lines through the origin and every affine line $G\subset \euc$ can be identified with $(\pi(G), G\cap \langle \pi(G)\rangle^{\perp})$ (see for instance \cite{Klain1997}).

Both the isometry group  of $ \euc $ and the isotropy group  of a fixed plane $ E \in \plans $ are unimodular groups; so the Haar measure of the group  of isometries is projected into a isometry-invariant measure $m$ on $ \mathcal {A} _ {3,2} $. 

For a measurable set $ B \subset \plans $ we consider 
$$
m(B)=\int_{\plans}\chi_{B}(E)dE:=\int_{\grass}\left(\int_{E_{0}^{\perp}}\chi_{B}(E_{0}+p)dp \right) d\nu
$$ where $ \chi_ {B} $ is the characteristic function of $ B $, $ dp $ denotes the ordinary Lebesgue measure on $ E_ {0} ^ {\perp} $ and $ d \nu $ a  normalized isometry-invariant measure  on $ \grass $ such that $ \nu (\grass) = 2 \pi $.

More generally, if $ f: \plans \rightarrow \R $ and $ \bar f: \grass \times \R ^ {3} \rightarrow \R $ are related  by $ \bar f (E_ {0}, p) = f (E_ {0} + p) $ we have
$$
\int _ {\plans} f (E)\, dE: = \int _ {\grass} \left (\int_ {E_ {0} ^ {\perp}} \bar f (E_ {0}, p) dp \right ) d \nu.
$$
Notice that the only measures on $ \plans $ invariant under isometries   are those of the form $ f (E) dE $ with $ f $ a constant function.

In a similar way one can  define a normalized isometry-invariant measure on $\rectes$ that will be denoted by $dG.$
\section{Invariant measures in the space of ordered pairs of planes}
We consider measures in the space $ \plans \times \plans $ of pairs of planes in $ \euc $   of the form $m_{\tilde{f}}:= \tilde f (E_ {1}, E_ {2}) dE_ {1}  dE_ {2 }. $ We want to study which functions $ \tilde f $ give an isometry-invariant measure, that is a measure $m_{\tilde{f}} $ satisfying
$ m _ {\tilde f} (B) = m _ {\tilde f} (gB) $ for every euclidean motion $ g  $. 
For instance, it is known that for a given compact convex set~$K$ one has $\int_{E\cap K\neq\emptyset}dE=M$. So when  $\tilde f (E_ {1}, E_ {2}) =1$ we have  
\begin{eqnarray}\label{17mars}
\int_{K\cap E_{i}\not =\emptyset}
dE_{1}dE_{2}=M^{2}=4\pi^{2}\mathcal{W}^{2},
\end{eqnarray}
where $M$ and $\mathcal{W}$ are the mean curvature  and  the mean width of $K$, respectively.

\begin{prop}\label{mesuresinv}
The measure given by $ \tilde f (E_ {1}, E_ {2}) dE_ {1} \, dE_ {2} $ in $ \plans \times \plans $ is invariant under isometries of $ \euc $ if and only if $ \tilde f (E_ {1}, E_ {2}) = f (\langle u_ {1}, u_ {2} \rangle) $ where $ \pi (E_ {i}) ^ {\perp} = \langle u_ {i} \rangle $, $i=1,2$ and $ f: [- 1,1] \rightarrow \R $ is an even measurable  function.
\end{prop}

\begin{proof}
Suppose that $ \tilde f (E_ {1}, E_ {2}) dE_ {1} \ dE_ {2} $ is invariant.~Using the representation of an element   $E\in \plans $ as a pair $(\pi(E), p)$  where $p=E\cap \langle\pi(E)\rangle^{\perp}$ we can write
$$\tilde f(E_{1},E_{2})=F(\pi(E_{1}),p_{1};\pi(E_{2}),p_{2})$$ 
for some $ F: (\grass \times \euc) ^ {2} \rightarrow \R $. For any translation $ \tau    $ it is $$ \tilde f (E_ {1} + \tau  , E_ {2} + \tau  ) = F (\pi (E_ { 1}), p_ {1} + \tau  ; \pi (E_ {2}), p_ {2} + \tau  ). $$ 
Due to the invariance of $ \tilde f (E_ {1} , E_ {2}) dE_ {1} \, dE_ {2} $ we have
$$
F(\pi(E_{1}),p_{1}+\tau  ;\pi(E_{2}),p_{2}+\tau  )=F(\pi(E_{1}),p_{1};\pi(E_{2}),p_{2})$$ 
and so $ F $ is independent of $ p_ {1} $ and $ p_ {2} $ and we can write $ \tilde f (E_ {1}, E_ {2}) = H (\pi (E_ {1}) , \pi (E_ {2})) $ for some function $ H $ on $ \grass \times \grass $.

Given $ t \in [-1,1] $ consider $ (V, W) \in \grass ^ {2} $ such that $ V = \langle v \rangle ^ {\perp} $, $ W = \langle w \rangle ^ {\perp} $ and $ t = \langle v, w \rangle $ with $v,w$ unit vectors. The function $ f (t) = H (V, W) $ is well defined since for any rotation $ \theta $ we have that $ H (\theta V, \theta W) = H (V, W ) $ and it is even. So it is proved  that there exists a measurable and even function  $ f: [- 1,1] \rightarrow \R $ such that $$ \tilde f (E_ {1}, E_ {2}) = f (\langle u_ {1}, u_ {2} \rangle). $$
If $ \tilde f $ is as above it is clear that $ \tilde f (E_ {1}, E_ {2}) dE_ {1}  dE_ {2} $ gives rise to an isometry-invariant measure.
\end{proof}




\section{Integral of functions  of pairs of planes meeting a convex set}
Let $K$ be a compact convex set in the euclidean space $E_{3}$. According to equality~\eqref{17mars} it is a natural question to evaluate
$$
\int_{E_{i}\cap K\not =\emptyset}\tilde f (E_ {1}, E_ {2}) dE_ {1} \ dE_ {2},
$$
where $\tilde f (E_ {1}, E_ {2}) dE_ {1} dE_ {2}$ is an isometry-invariant measure on $\plans\times \plans$. This can be done 
in terms of the coefficients of the expansion of the support function of $K$ in spherical harmonics and the coefficients of the Legendre series of the measurable even function $f:[-1,1]\rightarrow \R$ such that $\tilde f (E_ {1}, E_ {2})=f(\langle u_{1}, u_{2}\rangle)$ (see Proposition~\ref{mesuresinv}).

The following result is a special case, with a different notation, of Theorem 5 in~\cite{hugsch}. However we include a proof of this particular case for reader's convenience.  

\begin{thm}\label{2909b}Let $K$ be a compact convex set with   support function $p$  
given in terms of spherical  harmonics by \eqref{1607}.~Let $\tilde f (E_ {1}, E_ {2}) dE_ {1}\, dE_{2}$ be an isometry-invariant  measure on  $\plans\times \plans$ and $f:[-1,1]\rightarrow \R$ an even measurable function such that $\tilde f (E_ {1}, E_ {2})=f(\langle u_{1}, u_{2}\rangle)$ where $ \pi (E_ {i}) ^ {\perp} = \langle u_ {i} \rangle $, $i=1,2$. Then 
\begin{equation}\label{2909c}
\int_{E_{i}\cap K\not =\emptyset}\tilde f (E_ {1}, E_ {2}) dE_ {1}\, dE_ {2}={\lambda_{0}\over 4\pi}M^{2}+\sum_{\substack {n=2\\ n\, \mathrm{even}}}^{\infty}\lambda_{n}\lVert \pi_{n}(p)\rVert ^{2},
\end{equation}
where $\lambda_{n}=2\pi\int_{-1}^{1}f(t)P_{n}(t)\,dt$
with $P_{n}$
 the Legendre polynomial of degree $n$.
\end{thm}
\begin{proof}
As $\tilde f(E_{1},E_{2})=f(\langle u_{1}, u_{2}\rangle)$ we have that 
$$\int_{E_{i}\cap K\not =\emptyset}\tilde f (E_ {1}, E_ {2}) dE_ {1} \ dE_ {2} = \int_{E_{i}\cap K\not =\emptyset}f(\langle u_{1}, u_{2}\rangle)dE_ {1} \ dE_ {2}.$$
For a fixed plane $E_{2}$ in $ \plans$ and writing $dE_{i}=dp_{i}\wedge d\nu$ one has
\begin{multline}
\int_{E_{1}\cap K\not =\emptyset}\tilde f (E_ {1}, E_ {2}) dE_ {1}  =\int_{\grass}\left(\int_{\langle u_{1}\rangle}f(\langle u_{1}, u_{2}\rangle) dp_{1}\right)d\nu=\\ =\int_{\grass}f(\langle u_{1}, u_{2}\rangle) (p(u_{1})+p(-u_{1}))d\nu=\int_{S^{2}}f(\langle u, u_{2}\rangle) p(u)\, du,
\end{multline}
where the last equality follows from the fact that $f$ is even.

As $p(u)=\sum_{n=0}^{\infty}\pi_{n}(p)(u)$ (cf. \eqref{1607}) and using Funk-Hecke's theorem we have that
$$
\int_{E_{1}\cap K\not =\emptyset}\tilde f (E_ {1}, E_ {2}) dE_ {1}  = \sum_{n=0}^{\infty}\lambda_{n}\pi_{n}(p)(u_{2}).
$$
Notice that $f$ being even one has $\lambda_{n}=0$ for $n$ odd.
Now performing the integral with respect $E_{2}$ we have
\begin{multline*}
\int_{E_{i}\cap K\not =\emptyset}\tilde f (E_ {1}, E_ {2}) dE_ {1} \ dE_ {2} = 
\sum_{\substack {n=0\\ n\, \mathrm{even}}}^{\infty}\lambda_{n}\int_{\grass}\left(\int_{\langle u_{2}\rangle}\pi_{n}(p)(u_{2})dp_{2}\right)d\nu =\\
=  \sum_{\substack {n=0\\ n\, \mathrm{even}}}^{\infty}\lambda_{n}\int_{S^{2}}\pi_{n}(p)(u)p(u)\, du=\sum_{\substack {n=0\\ n\, \mathrm{even}}}^{\infty}\lambda_{n}\lVert \pi_{n}(p)\rVert ^{2},
\end{multline*}
where we have used the fact that $\pi_{n}(p)(u)=\pi_{n}(p)(-u)$ for $n$ even.
Taking into account that $\pi_{0}(p)=M/4\pi$ we have $$
\lVert \pi_{0}(p)\rVert ^{2}={M^{2}\over 16\pi^{2}}\lVert 1\rVert ^{2}={M^{2}\over 4\pi}
$$
and then
$$
\int_{E_{i}\cap K\not =\emptyset}\tilde f (E_ {1}, E_ {2}) dE_ {1} \ dE_ {2}={\lambda_{0}\over 4\pi}M^{2}+\sum_{\substack {n=2\\ n\, \mathrm{even}}}^{\infty}\lambda_{n}\lVert \pi_{n}(p)\rVert ^{2}.
$$
\end{proof}
\begin{ex}\label{mars2}
If $f(t)=\sqrt{1-t^{2}}$ then $f(\langle u_{1},u_{2}\rangle)=\sin(\theta_{12})$ where  $0\leq\theta_{12}\leq \pi$ is the angle between de planes $E_{1}$ and $E_{2}$ (that is, $\cos\theta_{12}=\pm\langle u_{1},u_{2}\rangle$ where $ \pi (E_ {i}) ^ {\perp} = \langle u_ {i} \rangle $, $i=1,2$).	
Applying Theorem \ref{2909b} with the corresponding coefficients
\begin{equation*}\label{lambdessinus}
\lambda_{2n}=2\pi\int_{-1}^{1}f(t)P_{2n}(t)=
-{\Gamma(n+\frac12)\Gamma(n-\frac12)\over n!(n+1)!}\frac{\pi}{2},\ \lambda_{0}=\pi^{2},\ \lambda_{2n+1}=0
\end{equation*}
(cf. \cite{grad}, {\bf 7.132}), one gets
\begin{equation}\label{lambdessinus}
\int_{E_{i}\cap K\neq\emptyset}\sin(\theta_{12})dE_{1}\,dE_{2}=\frac{\pi}{4}M^{2}-\frac{\pi}{2}\left(\sum_{n=1}^{\infty}{\Gamma(n+\frac12)\Gamma(n-\frac12)\over n!(n+1)!}\lVert \pi_{2n}(p)\rVert^{2}\right).
\end{equation}
\end{ex}
In the particular case that $f$ is a  Legendre polynomial  one obtains from Theorem~\ref{2909b} the following 

\begin{cor}\label{Pn}Let $K$ be a compact convex set with support function  $p$   
given in terms of spherical  harmonics by \eqref{1607}. Then if $P_{n}$ is the Legendre polynomial of even degree $n$, one has 
$$\int_{E_{i}\cap K\neq\emptyset}P_{n}(\langle u_{1},u_{2} \rangle)\,dE_{1}\,dE_{2}=\frac{4\pi}{2n+1}\lVert \pi_{n}(p)\rVert ^{2}.$$
\end{cor}
\begin{proof}
In this case $\lambda_{m}=0$ for $m\neq n$ and 
$\lambda_{n}=\dfrac{4\pi}{2n+1}$. 
\end{proof} 
\begin{ex}\label{2503b}
As the function $f(t)=t^{2n}$ can be  written in terms of Legendre polynomials as 
$\displaystyle t^{2n}=\sum_{k=0}^{n}\mu_{n,k}P_{2k}(t)$ with
\begin{equation}\label{t2n}
\mu_{n,k}={(4k+1)\Gamma(2n+1)\sqrt{\pi}\over 2^{2n+1}\Gamma(n-k+1)\Gamma(n+k+3/2)}
\end{equation}
(see \cite{grad}, {\bf 8.922}) we get the following consequence of Corollary \ref{Pn}:

\begin{eqnarray*}
\int_{E_{i}\cap K\neq\emptyset}\langle u_{1},u_{2} \rangle^{2n}\,dE_{1}\,dE_{2}= \sum_{k=0}^{n}{4\pi\over 4k+1}\mu_{n,k}\lVert \pi_{2k}(p)\rVert ^{2}
\end{eqnarray*}
where $\mu_{n,k}$ are given by \eqref{t2n}.
\end{ex}
To end this section we analyze equality \eqref{2909c} when  $K$
is a  convex set of constant width.
As said this means that $\pi_{n}(p)=0$ for  $n\neq 0$ even. 
%

\begin{prop}\label{Kampconst}Let $K$ be a compact convex set of constant width $\mathcal{W}$ and let $f:[-1,1]\longrightarrow \R$ an even bounded measurable function. Then
\begin{eqnarray}\label{15mars}\int_{E_{i}\cap K\neq\emptyset}f(\langle u_{1},u_{2}\rangle)\, dE_{1}\,dE_{2}=\lambda_{0}\pi \mathcal{W}^{2},
\end{eqnarray}
with $\lambda_{0}=2\pi \int_{-1}^{1}f(t)dt.$ Moreover  if the above equality holds when $f(t)=P_{2n}(t)$ where $P_{2n}$ is any Legendre polynomial of even degree $2n$, $n\not =0$  then $K$ is of constant width.
\end{prop}
\begin{proof}
Since $K$ is of constant width by \eqref{2909c} one gets 
$$\int_{E_{i}\cap K\neq\emptyset}f(\langle u_{1},u_{2}\rangle)\, dE_{1}\,dE_{2}={\lambda_{0}\over 4\pi}{M^{2}}$$
and remembering that $M=2\pi \mathcal{W}$ the equality follows. If  equality \eqref{15mars} holds for  $f(t)=P_{n}(t)$ with $n$ even, $n\not =0$, and  since the corresponding $\lambda_{0}$ vanishes one has $$\int_{E_{i}\cap K\neq\emptyset}P_{n}(\langle u_{1},u_{2} \rangle)\,dE_{1}\,dE_{2}=0.$$ Therefore by Corollary \ref{Pn} it follows  that $\lVert \pi_{n}(p)\rVert =0$ for every non zero $n$ even and $K$ is of constant width.
\end{proof} 


%
%
%
%

\section{Integrals of invariant measures  in terms of the  visual angle}
The aim of this section 
is to write  
the integral of a isometry-invariant measure over the pairs of planes meeting a convex set $K$, given in Theorem \ref{2909c},   as an integral of an appropriate function of the visual angle. 
\smallskip

Let us precise what we mean by the angle of a plane about a straight line $G$ and the visual angle of a convex set $K$ from a line $G$ not meeting $K$.

\pagebreak

\begin{defn}\label{0610}\mbox{}
\begin{enumerate}
\item[1.]  Given a  straight line $G$ let $(q; e_{1}, e_{2})$ be a fixed  affine orthonormal frame  in $G^{\bot}$ with $q\in G$. For each plane $E$ through $G$ let $u$ be the unit normal vector to $E$ pointing from the origin to it. Then the  angle $\alpha$ associated to~$E$
is the one given by 
$u=\cos(\alpha)e_{1}+\sin(\alpha)e_{2}$.
\item[2.] The visual angle  of a convex set $K$ from a line $G$ not meeting $K$ is the angle  $\omega=\omega(G)$, $0\leq \omega\leq\pi$, between the {\em half-planes} $E_{1},E_{2}$ through $G$ tangents to $K$. 
\end{enumerate}
\end{defn}
If $\alpha_{i}$ are the angles associated  to $E_{i}$, $i=1,2$ 
then $$\cos(\pi-\omega)=\cos(\alpha_{2}-\alpha_{1})=\langle u_{1},u_{2}\rangle$$ 
where $u_{1},u_{2}$ are  the normal unit vectors to $E_{1},E_{2}$ pointing from the origin, assuming the origin inside  $K$. 
\smallskip	

The measure $ dE_ {1} \, dE_ {2 }$ in the space $ \plans \times \plans $ of pairs of planes in $ \euc $ can be written according to Santaló (cf. \cite{santalo}, section II.12.6) as
\begin{eqnarray}\label{01}
dE_{1}\, dE_{2}=\sin^{2}(\alpha_{2}-\alpha_{1})\,d\alpha_{1}\, d\alpha_{2}\, dG.\end{eqnarray}
Then we can prove the following
\begin{thm}\label{1310c} Let $K$ be a compact convex set 
and  let $f:[-1,1]\longrightarrow \R$ be an even   continuous function. Let $H$ be the  ${\mathcal{C}}^{2}$ function on $[-\pi,\pi]$
 satisfying 
 $$H''(x)=f(\cos(x))\sin^{2}(x), \quad  -\pi\leq x\leq \pi,  \quad  H(0)=H'(0)=0.$$
Then \begin{equation}\label{mars}
\int_{E_{i}\cap K\neq\emptyset}f(\langle u_{1},u_{2} \rangle)\,dE_{1}\,dE_{2}=
{\pi}H(\pi)F+
2\int_{G\cap K=\emptyset} H(\omega)\,dG,
\end{equation}
where $u_{i}$ are  normal unit vectors   to the planes $E_{i}$, $\omega=\omega(G)$  is the visual angle from the line $G$ and $F$ is the area of the boundary of $K$. 
\end{thm}
\begin{proof}
Let $G=q+\langle u\rangle $ with $u$ a unit director vector such that $K\cap G =\emptyset$. Let $E_{i}, i=1,2$ be the supporting planes of $K$  through $G$. Take now an affine orthonormal  frame $\{q;e_{1},e_{2}, u\}$ in $\euc$ such that $E_{1}=q+\langle e_{1},u\rangle .$ Every plane $E$ through $G$ can be written as $E=q+\langle v_{\alpha},u\rangle $ where $v_{\alpha}=\cos\alpha e_{1}+\sin\alpha e_{2}$ with $\alpha \in [0,\pi)$ and the planes $E$ intersecting $K$ correspond to angles $\alpha\in [0,\omega(G)].$ Then using \eqref{01} one has 
\begin{multline*}
\int_{E_{i}\cap K\neq\emptyset}f(\langle u_{1},u_{2} \rangle)\,dE_{1}\,dE_{2}= \\ = 
\int_{G\cap K=\emptyset}\int_{0}^{\omega}\int_{0}^{\omega}H''(\alpha_{2}-\alpha_{1})\,d\alpha_{1}\,d\alpha_{2}\,dG\, + \\ + 
\int_{G\cap K\neq\emptyset}\int_{0}^{\pi}\int_{0}^{\pi}H''(\alpha_{2}-\alpha_{1})\,d\alpha_{1}\,d\alpha_{2}\,dG.  
\end{multline*}
Evaluating the inner integrals and taking into account that $\int_{G\cap K\neq\emptyset}dG=\frac{\pi}{2}F$
it follows
\begin{multline*}
\int_{E_{i}\cap K\neq\emptyset}f(\langle u_{1},u_{2} \rangle)\,dE_{1}\,dE_{2}= \\ =
\frac12{\pi}(H(\pi)+H(-\pi))F + 
\int_{G\cap K=\emptyset} (H(\omega)+H(-\omega))\,dG.
\end{multline*}
Since  $H(0)=H'(0)=0$ and $H''(x)=H''(-x)$ it is easy to see that $H(x)=H(-x)$ and the result follows.
\end{proof}
%


\begin{ex}\label{exblas}
Consider $f(t)=1$ then $H(x)=(x^{2}-\sin^{2}x)/4$ and
\begin{eqnarray}\label{2503}
M^{2}=\int_{E_{i}\cap K\neq\emptyset}dE_{1}\,dE_{2}=
\frac14{\pi}^{3}F+
\frac12\int_{G\cap K=\emptyset} (\omega^{2}-\sin^{2}\omega)\,dG,
\end{eqnarray}
which gives the well known Blaschke's formula (cf. \cite{Blaschke2007})
\end{ex}

\begin{ex}
Let $f(t)=\sqrt{1-t^{2}}$ be the function considered in Example \ref{mars2}.  In  this case the corresponding 
 function $H$
 such that
 $$H''(x)=f(\cos(x))\sin^{2}(x)=|\sin^{3}(x)|$$
 is 
 given by
 $$H(x)=\frac23\left(|x|-|\sin x|\right)-\frac19|\sin^{3} x|.$$ 
 Now, since $\omega\in[0,\pi]$, Theorem \ref{1310c} and equality  \eqref{lambdessinus} leads to
\begin{multline*}
\int_{G\cap K=\emptyset}\left(\omega-\sin\omega-\frac1{3!}\sin^{3}\omega\right)dG =\\ =
\pi\left(\frac{3M^{2}}{16}- \frac{1}{2}\pi F
-\frac{3}{8}\sum_{n=1}^{\infty}{\Gamma(n+\frac12)\Gamma(n-\frac12)\over n!(n+1)!}\lVert\pi_{2n}(p)\rVert^{2}\right).
\end{multline*}
\end{ex}

\begin{ex}\label{sin4}
Taking now $f(t)=t^{2}$, one gets $H(x)=1/16 (x^{2}-\sin^{2}x\cos^{2}x)=1/16(x^{2}-\sin^{2}x+\sin^{4}x)$. According Exemple \ref{2503b}, Theorem \ref{1310c} and Blaschke formula one gets
\begin{multline}
{1\over 3}M^{2}+{8\pi\over 15}\lVert \pi_{2}(p)\rVert ^{2} =
\int_{E_{i}\cap K\neq\emptyset}\langle u_{1}, u_{2}\rangle^{2}\, dE_{1}\,dE_{2}= \\ =
\frac14{M^{2}}+\frac18 \int_{G\cap K=\emptyset}\sin^{4}\omega\, dG,
\end{multline}
and so
$$
\int_{G\cap K=\emptyset}\sin^{4}\omega\, dG= \frac23 M^{2}+{64\pi\over 15}\lVert \pi_{2}(p)\rVert ^{2}.
$$
Notice that in the case that $K$ has constant width one obtains $$\int_{G\cap K=\emptyset}\sin^{4}\omega\, dG= \frac23 M^{2}.$$
\end{ex}
More generally if $f(t)=t^{2n}$ we have the following
\begin{prop} Let $K$ be a compact convex set with mean curvature  $M.$ Then one has
\begin{equation}\label{prod2n}
\int_{E_{i}\cap K\neq\emptyset}\langle u_{1},u_{2} \rangle^{2n}\,dE_{1}\,dE_{2}=
-4A_{1}^{(n)}M^{2}+2\sum_{r=2}^{n+1}A_{r}^{(n)}\int_{G\cap K=\emptyset}\sin^{2r}\omega\,dG,
\end{equation}
where
$$
A_{r}^{(n)}={(-1)^{r}{n\choose r-1}\over 4r^{2}}\, {}_3 F_2(1,r+1/2,r-n-1;r,r+1;1),\quad r=1,\dots,n+1.$$
\end{prop}
We recall that
${}_3 F_2$ is the hypergeometric function given by
$$
{}_3 F_2(a_{1},a_{2},a_{3};b_{1},b_{2};z)= \sum_{n=0}^{\infty}{(a_{1})_{n}(a_{2})_{n}(a_{3})_{n}\over (b_{1})_{n}(b_{2})_{n}}{z^{n}\over n!}
$$
with  $(a)_{n}=\Gamma(a+n)/\Gamma(a)$ is the Pochhammer symbol.

The integral in the left-hand side of \eqref{prod2n} was calculated in Exemple \ref{2503b} in terms of the expansion  in spherical harmonics of the support function of $K$. 
\begin{proof}
We will apply Theorem \ref{1310c} for $f(t)=t^{2n},$ that is $H''(x)=\cos^{2n}x\, \sin^{2}x.$
We search the corresponding $H(x).$ Le us write
$$
H''(x)=\sum_{k=0}^{n}{n\choose k}(-1)^{k}\sin^{2(k+1)}x.
$$
Integrating twice  and using formula {\bf 2.511}(2) in \cite{grad} we get
\begin{multline*}
H(x)=\frac14(x^{2}-\sin^{2}x)+\\ +\sum_{k=1}^{n}{n\choose k}(-1)^{k}\left[
\alpha_{k}{x^{2}\over 2} -{1\over 2(k+1)}\left({\sin^{2(k+1)}x\over 2(k+1)}+\sum_{r=1}^{k}\beta_{r}^{(k)}{\sin^{2(k-r+1)}x\over 2(k-r+1)}\right)
\right],
\end{multline*}
with
$$
\alpha_{k}={\Gamma(k+3/2)\over\Gamma(k+2)},\quad
\beta_{r}^{(k)}= {\Gamma(k-r+1)\Gamma(k+3/2)\over \Gamma(k+1)\Gamma(k-r+3/2)}.
$$
This provides for $H(x)$ an expression of the type
$$H(x)=A_{0}^{(n)}x^{2}+A_{1}^{(n)}\sin^{2}x+\cdots +A_{n+1}^{(n)}\sin^{2(n+1)}x.$$
Let us find now the values on $A_{r}^{(n)}$. Beginning with $A_{0}^{(n)}$ 
one gets
$$
A_{0}^{(n)}={{2n\choose n}\over 2^{2(n+1)}}.
$$
For $r>0$ one has
$$
A_{r}^{(n)}=\frac14 \sum_{k=r-1}^{n}{n\choose k}{{(-1)}^{k+1}\over k+1}{\beta_{k-r+1}^{(k)}\over r},\quad  r=1,\dots, n+1,
$$
or in a more compact form
$$
A_{r}^{(n)}={(-1)^{r}{n\choose r-1}\over 4r^{2}}\, {}_3 F_2(1,r+1/2,r-n-1;r,r+1;1),\quad r=1,\dots,n+1.
$$
One easily sees that  $A_{0}^{(n)}=-A_{1}^{(n)}$ 
an so
\begin{equation}
H(x)=A_{0}^{(n)}(x^{2}-\sin^{2}x)+\sum_{r=2}^{n+1}A_{r}^{(n)}\sin^{2r}x.
\end{equation}
Finally Theorem \ref{1310c} and Blaschke identity \eqref{2503} give the result.
\end{proof}

\subsection{Crofton's formula in the space}
 Blaschke's formula \eqref{2503} reveals the significance of the function 
 of the visual  angle $\omega^{2}-\sin^{2}\omega$. One can ask what role the function $\omega-\sin\omega$  plays; 
 this function, interpreting $\omega$ as the visual angle in the plane,  is significant thanks to Crofton's formula 
 $\int_{P\notin K}(\omega-\sin\omega) dP=L^{2}/2- \pi F$, where $K$ is a compact convex set in the plane with area $F$ and length of its boundary $L$ (see \cite{santalo}).

  In \cite[p. 85]{Blaschke2007} Blaschke shows  that 
\begin{equation}\label{pepitogrillo}
\int_{G\cap K=\emptyset}(\omega-\sin\omega)dG=\frac14 \int_{u\in S^{2}}L_{u}^{2}du-{\pi^{2}\over 2}F,
\end{equation}
where $L_{u}$ is the length of the boundary of  the projection of $K$ on $\langle u\rangle^{{\bot}}$.

It can be easily seen that $\int_{u\in S^{2}}L_{u}du=2\pi M$ and from this equality, applying Schwarz's inequality, one gets
\begin{equation}\label{desblasc}
\int_{u\in S^{2}}L_{u}^{2}du\geq \pi M^{2}.
\end{equation}
Introducing \eqref{desblasc} into \eqref{pepitogrillo} one obtains 
\begin{equation}\label{mars3}
\int_{G\cap K=\emptyset}(\omega-\sin\omega)dG\geq {\pi\over 4}(M^{2}-2\pi F).
\end{equation}

As a consequence of Theorem \ref{1310c} we can now evaluate the deficit in both inequalities  \eqref{desblasc} and \eqref{mars3}. 

 \begin{thm}\label{1903} Let $K$ be a compact convex set with support function $p$, area of its boundary $F$ and mean curvature $M$.
Let $L_{u}$ be the length of the boundary of  the projection of $K$ on $\langle u\rangle^{{\bot}}$
 and let $\omega=\omega(G)$ be the visual angle of $K$ from the line~$G$. Then
\begin{enumerate}[i)]
\item $\displaystyle \int_{u\in S^{2}}L_{u}^{2}du=\pi M^{2}+4\pi\sum_{n=1}^{\infty}{\Gamma(n+1/2)^{2}\over \Gamma(n+1)^{2}}\lVert \pi_{2n}(p)\rVert^{2},$
\item $\displaystyle \int_{G\cap K=\emptyset}(\omega-\sin\omega)dG={\pi\over 4}(M^{2}-2\pi F)+\pi\sum_{n=1}^{\infty}{\Gamma(n+1/2)^{2}\over \Gamma(n+1)^{2}}\lVert \pi_{2n}(p)\rVert^{2}.$
\end{enumerate}

Moreover equality holds both in \eqref{desblasc} and \eqref{mars3} if and only if $K$ is of constant width. 
 \end{thm}
\begin{proof}
We consider  $f(t)=1/\sqrt{1-t^{2}}$. For this function the corresponding  $H$ in Theorem \ref{1310c}  is $H(x)=|x|-|\sin x|.$ Applying  equality \eqref{mars} and Theorem \ref{2909b} with the corresponding $\lambda_{2n}$'s given by 
$$
\lambda_{2n}=2\pi \int_{-1}^{1}f(t)P_{2n}(t)dt= 2\pi{\Gamma(n+1/2)^{2}\over \Gamma(n+1)^{2}},
$$
(cf. \cite{grad}, {\bf 7.226}), item ii)  follows. Equality i) is a consequence of ii) and \eqref{pepitogrillo}. 

The statement about equality in \eqref{desblasc} and \eqref{mars3} is a consequence of the fact that $K$ is of constant width if and only if $\pi_{2n}(p)=0$ for $n\not =0.$
%
 \end{proof}

\section{A formulation with Fourier series}
In this section we give an alternative formulation of Theorem \ref{1310c} in terms of Fourier coefficients of the function $H''(x)$.  Since $f$ is even one has that $H''(x)=f(\cos(x))\sin^2(x)$ is an even  $\pi$-periodic  function. Let
\begin{equation}\label{coeffa2n}
H''(x)=\frac12 a_{0}+\sum_{n\geq 1}a_{2n}\cos(2nx)
\end{equation} 
be the Fourier expansion of $H''(x).$ Integrating twice and taking into account that  $H(0)=H'(0)=0$ one obtains
\begin{eqnarray}\label{1510d}H(x)=\frac{a_{0}}{4}x^{2}+\sum_{n\geq 1}\frac{a_{2n}}{4n^{2}}(1-\cos(2nx)).\end{eqnarray}
Using this expression of the function $H$, Theorem \ref{1310c} can be written 
as
\begin{prop}Let $K$ be a compact convex set  and  let $f:[-1,1]\longrightarrow \R$ be an even continuous function. Let $H$ be the ${\mathcal C}^{2}$ function on $[-\pi,\pi]$
 satisfying 
\begin{eqnarray}\label{1510q}H''(x)=f(\cos(x))\sin^{2}(x),\quad -\pi\leq x\leq\pi,  \quad  H(0)=H'(0)=0.\end{eqnarray}
If $H(x)$ is given by \eqref{1510d} then 
\begin{multline}\label{1310bb}
\int_{E_{i}\cap K\neq\emptyset}f(\langle u_{1},u_{2} \rangle)\,dE_{1}\,dE_{2}=\\ =
\frac{a_{0}}{4}\pi^{3}F +
\frac12\int_{G\cap K=\emptyset} \bigg({a_{0}}\omega^{2}+\sum_{n\geq 1}\frac{a_{2n}}{n^{2}}(1-\cos(2n\omega))\bigg)\,dG,
\end{multline}
where $u_{i}$ are  normal vectors   to the planes $E_{i}$, the visual angle from the  line $G$ is $\omega$  and $F$ denotes the area of the boundary of $K$.  
\end{prop}
The right hand side of  \eqref{1310bb} can be written as a linear combination of integrals of even powers of $\sin\omega$. For this purpose we will use Blaschke formula \eqref{2503} and the known equality 
\begin{equation}\label{cos2n}
\cos 2n x=\sum_{m=0}^{n}\alpha_{n,m}\sin^{2m}x\quad \mathrm{with}
\quad
\alpha_{n,m}=\frac{(-1)^{m}\,n\,2^{2m}(n+m-1)!}{(2m)!(n-m)!},
\end{equation}
which follows easily from the equality $\cos(2nx)=(-1)^{n}T_{2n}(\sin x)$ where $T_{2n}$ is Chebyshev's polynomial of degree $2n$. We can state
\begin{prop}\label{intsinus2}
Let $K$ be a compact convex set  and  let $f:[-1,1]\longrightarrow \R$ be an even continuous function. Let $H$ be the  ${\mathcal C}^{2}$ function on $[-\pi,\pi]$
 satisfying 
\begin{eqnarray*}H''(x)=f(\cos(x))\sin^{2}(x), \quad -\pi\leq x\leq \pi, \quad  H(0)=H'(0)=0.\end{eqnarray*}
If $H(x)$ is given by \eqref{1510d} then 
\begin{multline}\label{intfpotsinus}
\int_{E_{i}\cap K\neq\emptyset}f(\langle u_{1},u_{2} \rangle)\,dE_{1}\,dE_{2}=\\ =
a_{0}M^{2}- 
\frac12\sum_{m=2}^{\infty}\left(\sum_{n=m}^{\infty}{a_{2n}\over n^{2}}\alpha_{n,m}\int_{G\cap K=\emptyset} \sin^{2m}\omega\, dG\right),
\end{multline}
where $u_{i}$ are  normal vectors   to the planes $E_{i}$, the visual angle from the  line $G$ is $\omega$, $F$ denotes the area of the boundary of $K$, the coefficients $\alpha_{n,m}$ are given by \eqref{cos2n} and the coefficients $a_{2n}$ by \eqref{coeffa2n}.
\end{prop}
\begin{proof}
Using \eqref{cos2n} the right hand side of \eqref{1310bb} is written as
\begin{multline*}
\frac{a_{0}\pi^{3}}{4}F +
\frac12\int_{G\cap K=\emptyset} \bigg({a_{0}}\omega^{2}-\sum_{n=1}^{\infty}{a_{2n}\over n^{2}}\alpha_{n,1}\sin^{2}\omega 
-\sum_{n=2}^{\infty}{a_{2n}\over n^{2}}\sum_{m=2}^{n}\alpha_{n,m}\sin^{2m}\omega
\bigg)\,dG =\\ =
\frac{a_{0}\pi^{3}}{4}F +
\frac12\int_{G\cap K=\emptyset} \bigg({a_{0}}(\omega^{2}-\sin^{2}\omega)
-\sum_{n=2}^{\infty}{a_{2n}\over n^{2}}\sum_{m=2}^{n}\alpha_{n,m}\sin^{2m}\omega
\bigg)\,dG
\end{multline*}
where we have used that $\alpha_{n,0}=1, \alpha_{n,1}=-2n^{2} $ and $a_{0}=-2\sum_{n=1}^{\infty}a_{2n}$ which is a consequence of the fact that $H''(0)=0.$ Using Blaschke formula \eqref{2503} and reordering the double sum  the result follows.
\end{proof}

\subsection{A basis for the integrals of invariant measures}

As a consequence of Proposition \ref{intsinus2} we can exhibit
 a simple family of polynomial functions that are in some sense a basis for  the integrals  in Theorem \ref{2909b}.
Consider the polynomials \begin{eqnarray}\label{2403}h_{m}(t)=m(2mt^{2}-1)(1-t^{2})^{m-2},\ m> 1.\end{eqnarray}
Then for $H''(x)=h_{m}(\cos(x))\sin^{2}(x)$
one easily checks
that
$H(\omega)=\frac{1}{2}\sin^{2m}\omega$
and  Theorem \ref{1310c} applied to $h_{m}(t)$ gives
$$\int_{E_{i}\cap K\neq\emptyset}h_{m}(\langle u_{1},u_{2} \rangle)\,dE_{1}\,dE_{2}=\int_{G\cap K=\emptyset}\sin^{2m}\omega\, dG,$$ 
that together with equation \eqref{intfpotsinus}
leads to the following
\begin{prop}\label{2403b}Under the same hypotheses and notation as  in Proposition \ref{intsinus2} one has
\begin{multline*}
\int_{E_{i}\cap K\neq\emptyset}f(\langle u_{1},u_{2} \rangle)\,dE_{1}\,dE_{2}=\\ =
a_{0}M^{2}- 
\frac12\sum_{m=2}^{\infty}\left(\sum_{n=m}^{\infty}{a_{2n}\over n^{2}}\alpha_{n,m}\int_{E_{i}\cap K\neq\emptyset} h_{m}(\langle u_{1},u_{2}\rangle)\,dE_{1}\,dE_{2}\right),
\end{multline*}
where the polynomials $h_{m}$ are given in \eqref{2403}.
\end{prop}
So every invariant integral can be written as an infinite linear combination of the integrals of the invariant measures given by  the polynomials $h_{m}$.
\section{Powers of sine function of the visual angle}
Equality \eqref{intfpotsinus} suggests to consider integrals of the form $$\int_{G\cap K\neq \emptyset}\sin^{2m}\omega\, dG.$$
Integrals of the power of the sine of the visual angle $\omega$ for a compact convex set $K$ in the plane  were  considered in  \cite{Cufi2019} obtaining
$$
\int_{P\not \in K}\sin^{2m}\omega\ dP=\lambda_{0}L^{2}+\sum_{\substack {k\geq 2,\\ \mathrm{even}}}\lambda_{k}c_{k}^{2},
$$
for convenient $\lambda_{k}$ and where $c_{k}$ depends on the support function of $K$. Applying this formula to the projections $K_{u}$ of a compact convex set $K$ in the euclidean space on the plane $\langle u\rangle^{\perp}$ and taking into account that $dG=dP\wedge du$ and $\omega_{u}(P)=\omega(G)$ one gets 
\begin{multline*}
\int_{G\cap K=\emptyset} \sin^{2m}\omega\,dG=
\frac12\int_{u\in S^{2}}\int_{P\not\in K_{u}}\sin^{2m}\omega_{u}\,dP\, du=\\
=\frac{\lambda_{0}}2\int_{u\in S^{2}}L_{u}^{2}du+\sum_{\substack {k\geq 2,\\ \mathrm{even}}}{\lambda_{k}\over 2}\int_{u\in S^{2}}c_{k,u}^{2}du.
\end{multline*}
The first integral  in the right hand side was geometrically interpretated in Theorem~\ref{1903} but the second one has not been handled.
\smallskip	

A more direct way to deal with the integral of the power of the sine of the visual angle for a compact convex set $K$ in the euclidean space is by means of  our previous results. In fact, we have the following
\begin{prop}\label{potsin}
Let $K$ be a compact convex set with   support function $p$  
given in terms of spherical  harmonics by \eqref{1607} then
\begin{equation}\label{sinpow}
\int_{G\cap K=\emptyset} \sin^{2m}\omega\,dG={m\sqrt{\pi}(m-2)!\over 4\Gamma(m+\frac12)}M^{2}+\sum_{k=1}^{m-1}\beta_{2k}\lVert\pi_{2k}(p)\rVert^{2},\quad m>1,
\end{equation}
where $M$ is the mean curvature of $K$ and the $\beta_{2k}$'s are given by
\begin{equation}\label{lambdasub2k}
\beta_{2k}=(-1)^{k+1}	{m(m-2)!^{2}\Gamma(k+1/2)\big((2m-1)(2k-1)(k+1)+m\big)\over
k!(m-k-1)!\Gamma(m+k+1/2)
}\pi.
\end{equation}
\end{prop}
\begin{rem}
When $K$ is a convex set of constant width one has $\lVert\pi_{2k}(p)\rVert^{2}=0$ and so
\begin{equation*}
\int_{G\cap K=\emptyset} \sin^{2m}\omega\,dG={m\sqrt{\pi}(m-2)!\over 4\Gamma(m+\frac12)}M^{2}.
\end{equation*}
\end{rem} 
\begin{proof}
Comparing relations  \eqref{2909b} and  \eqref{mars} we have
$$
{\beta_{0}\over 4\pi}M^{2}+\sum_{k=1}^{\infty}\beta_{2k}\lVert \pi_{2k}(p)\rVert ^{2}={\pi}H(\pi)F+
2\int_{G\cap K=\emptyset} H(\omega)\,dG
$$
where $H''(x)\!=\!f(\cos x)	\sin^{2}x$ with $H(0)\!=\!H'(0)=0$ and $\beta_{2k}\!=\!2\pi\int_{-1}^{1}f(t)P_{2k}(t)\,dt.$
As said above, the function $H$ corresponding to the 
 polynomials $h_{m}$ given in \eqref{2403}
is  $H(\omega)=\frac12\sin^{2m}\omega$. 
In this case we have that $\beta_{2k}=0$ for $k>m-1$ because $h_{m}(t)$ is a polynomial of degre $2(m-1)$ and $\int_{-1}^{1}h_{m}(t)P_{2k}(t)=0$ for $k>m-1$. 
In order to compute 
$$\beta_{2k}=2\pi\int_{-1}^{1}h_{m}(x)P_{2k}(x)dx$$
we first observe that, using a Computer algebra system (CAS)
we get
\begin{eqnarray}\label{0101}
\int_{-1}^{1}h_{m}(x)x^{2j}dx=
{m(4mj-2j+1)\Gamma(j+1/2)\Gamma(m-1) \over 2 \Gamma(m+j+1/2)}.
\end{eqnarray}
%
Now, recalling the expression of the Legendre polynomials 
$$P_{2k}(x)=\frac{1}{2^{2k}}\sum_{r=0}^{k}(-1)^{r}{2k\choose r} {4k-2r\choose 2k} x^{2(k-r)}$$
we have  
\begin{multline*}
 \int_{-1}^{1}h_{m}(x)P_{2k}(x)dx\ =
\\=\frac{m\Gamma(m-1)}{2^{2k+1}}\sum_{r=0}^{k}(-1)^{r}\!{2k\choose r} 
\!{4k-2r\choose 2k}\frac{\Gamma(k\!-\!r+1/2)(4m(k\!-\!r)\!-\!2(k-r)\!+\!1)}{\Gamma(m+k-r+1/2)}, 
\end{multline*}
but this sum is computable,  using again a CAS we have 
\begin{multline*}
 \int_{-1}^{1}h_{m}(x)P_{2k}(x)dx\ =
\\  =
\frac{1}{2^{2k+1}}m\Gamma(m-1)\bigg( \Gamma(m-1)\Gamma(1/2-2k)\Gamma(k-1/2)\cdot\\ \cdot
\bigg[(m-1)(2k-1){4k\choose 2k}((4m-2)k+1)-2(2m-1)k(4k-1){4k-2\choose 2k}(2m+2k-1) \bigg]\cdot\\ \cdot \frac{1}{2\Gamma(1/2-k)\Gamma(m-k)\Gamma(m+k+1/2)}
 \bigg).
\end{multline*}
Simplifyng this expression and taking into account that 
$\Gamma(1/2+j)\Gamma(1/2-j)=(-1)^{j}\pi$
we finally obtain
\begin{eqnarray*}
\beta_{2k}&=&2\pi\int_{-1}^{1}h_{m}(x)P_{2k}(x)dx  \\
&=&(-1)^{k+1}\frac{m(m-2)!^{2}\Gamma(k+1/2)\big((2m-1)(2k-1)(k+1)+m\big)}
{k!(m-k-1)!\Gamma(m+k+1/2)}\pi.
\end{eqnarray*}

\end{proof}

\bibliographystyle{plain}
\bibliography{CGRdP}

\end{document}